\documentclass[11pt]{amsart}

\usepackage{latexsym}
\usepackage{amssymb,amsfonts,amsmath,amsthm}
\usepackage{enumerate}
\usepackage{mdwlist}
\usepackage[margin=1.3in]{geometry}
\usepackage[all]{xy}

\newtheorem{theorem}{Theorem}  

\newtheorem{lemma}[theorem]{Lemma}

\newtheorem{proposition}[theorem]{Proposition}
\newtheorem{conjecture}[theorem]{Conjecture}

 {\begin{list}{}%
         {\setlength{\leftmargin}{#1}}%
         \item[]%
 }
 {\end{list}}

\def\COMMENT#1{}
\def\TASK#1{}

\numberwithin{theorem}{section}
\numberwithin{equation}{section}

\newdimen\margin   
\def\textno#1&#2\par{%
   \margin=\hsize
   \advance\margin by -4\parindent
          \setbox1=\hbox{\sl#1}%
   \ifdim\wd1 < \margin
      $$\box1\eqno#2$$%
   \else
      \bigbreak
      \hbox to \hsize{\indent$\vcenter{\advance\hsize by -3\parindent
      \it\noindent#1}\hfil#2$}%
      \bigbreak
   \fi}
\def\noproof{{\unskip\nobreak\hfill\penalty50\hskip2em\hbox{}\nobreak\hfill%
       $\square$\parfillskip=0pt\finalhyphendemerits=0\par}\goodbreak}
\def\endproof{\noproof\bigskip}

\title[Bipartitions of highly connected tournaments]{Bipartitions of highly connected tournaments}

\author{ Jaehoon Kim, Daniela K\"uhn, Deryk Osthus }
\thanks{ The research leading to these results was partially supported by the European Research Council under the European Union's Seventh Framework Programme (FP/2007--2013) / ERC Grant Agreements no. and 306349 (J.~Kim and D.~Osthus) as well as 258345 (D.~K\"uhn).}

\begin{document}

\begin{abstract}
We show that if $T$ is a strongly $10^9k^6\log(2k)$-connected tournament, 
there exists a partition $A, B$ of $V(T)$ such that each of $T[A]$, $T[B]$ and $T[A,B]$ is strongly $k$-connected.
This provides tournament analogues of two partition conjectures of Thomassen regarding highly connected graphs.
\end{abstract}

\date{\today}

\maketitle 

\section{Introduction}

\subsection{Partitions of highly connected tournaments}
The study of graph partitions where the resulting subgraphs inherit the properties of the original graph has a long history
with some surprises and numerous open problems, see e.g.~the survey~\cite{Scott}.
For example, a classical result of Hajnal~\cite{Haj} and Thomassen~\cite{Thom} implies that for every $k$ there exists an integer $f(k)$ such that every $f(k)$-connected graph has a vertex partition into sets $A$ and $B$ so that both $A$ and $B$ induce $k$-connected graphs. 
A related conjecture of Thomassen~\cite{CTbip} states that for every $k$ there is an $f(k)$ such that every
$f(k)$-connected graph $G$ has a bipartition $A,B$ so that the spanning bipartite graph $G[A,B]$
is $k$-connected.
It is not hard to show that one cannot achieve both the above properties simultaneously in a highly connected graph.
On the other hand, our main result states that for tournaments, we can find a single partition which achieves both the above properties.
Below we denote by $T[A,B]$ the bipartite subdigraph of $T$ which consists of all edges between $A$ and $B$ but no others.

\begin{theorem}\label{main theorem}
Let $T$ be a tournament and $k \in \mathbb{N}$. If $T$ is strongly $10^9k^6\log(2k)$-connected,
there exists a partition $V_1$, $V_2$ of $V(T)$ such that each of $T[V_1]$, $T[V_2]$ and $T[V_1,V_2]$ is strongly $k$-connected.
\end{theorem} 

We have made no attempt to optimize
the bound on the connectivity in Theorem~\ref{main theorem}. 
(It would be straightforward to obtain minor improvements at the expense of more careful calculations.)
On the other hand,
it would be interesting to obtain the correct order of magnitude for the connectivity bound.

K\"uhn, Osthus and Townsend~\cite{KOT} earlier proved the weaker result that every strongly $10^8k^6\log (4k)$-connected tournament $T$
has a vertex partition $V_1, V_2$ such that $T[V_1]$ and $T[V_2]$ are both strongly $k$-connected (with some control over the sizes of $V_1$
and $V_2$). This proved a conjecture of Thomassen. \cite{KOT} raised the question whether this can be extended to digraphs.

As described later, our proof of Theorem~\ref{main theorem} develops ideas in~\cite{KOT}. These in turn are based on the
concept of robust linkage structures which were introduced in~\cite{KLOP} to prove a conjecture of Thomassen on edge-disjoint
Hamilton cycles in highly connected tournaments. Further (asymptotically optimal) results leading on from these approaches were obtained
by Pokrovskiy~\cite{P, P2}.

\subsection{Subdivisions and linkages}
The famous Lov\'asz path removal conjecture states that for every $k\in \mathbb{N}$ there exists $g(k)\in \mathbb{N}$ such that for every pair $x,y$ of vertices in a 
$g(k)$-connected graph $G$ we can find an induced path $P$ joining $x$ and $y$ in $G$ for which $G\setminus V(P)$ is $k$-connected.
In~\cite{CTt}, Thomassen proved a tournament version of this conjecture. 
Here, we generalize his argument to observe that highly connected tournaments contain a non-separating subdivision of any given digraph $H$
(with prescribed branch vertices). The case when $d=2$ and $m=1$ corresponds to the
result in~\cite{CTt}.

\begin{theorem}\label{thm:subd}
Let $k,d,m \in \mathbb{N}$.
Suppose that $T$ is a strongly $(k+m(d+2))$-connected tournament, that $D$ is a set of $d$ vertices in $T$, that $H$ is a digraph on $d$ vertices and $m$ edges and that $\phi$ is a bijection from $V(H)$ to $D$. Then $T$
contains a  subdivision $H^*$ of $H$ such that 
\begin{itemize}
\item[(i)] for each $h \in V(H)$ the branch vertex of $H^*$ corresponding to $h$ is $\phi(h)$,
\item[(ii)] $T\setminus V(H^*)$ is strongly $k$-connected,
\item[(iii)] for every edge $e$ of $H$, the path $P_e$ of $H^*$ corresponding to $e$ is backwards-transitive.
\end{itemize}
\end{theorem}
Here a directed path $P=x_1 \dots x_t$ in a tournament $T$ is \emph{backwards-transitive}
if $x_ix_j$ is an edge of $T$ whenever $i \ge j+2$.
The graph version of Theorem~\ref{thm:subd} is still open and would follow from the following
beautiful conjecture of Thomassen~\cite{CTdecsubd}.
\begin{conjecture} \label{partitionconj}
For every $k\in\mathbb{N}$ there exists $f(k)\in\mathbb{N}$ such that if 
$G$ is a $f(k)$-connected graph and $M\subseteq V(G)$ consists of $k$
vertices then there exists a partition $V_1$, $V_2$ of $V(G)$ such that $M\subseteq V_1$, both $G[V_1]$ and $G[V_2]$ are 
$k$-connected, and each vertex in $V_1$ has at least $k$ neighbours in $V_2$.
\end{conjecture}
The case $|M|=2$ would already imply the path removal conjecture. The case $M=\emptyset$ was proved in~\cite{KO}. It
implies the existence of non-separating subdivisions (without prescribed branch vertices) in highly connected graphs.
Clearly, Theorem~\ref{main theorem} implies a tournament version of Conjecture~\ref{partitionconj}.

The next theorem guarantees a spanning linkage in a highly connected tournament.
It was proved by Thomassen~\cite{CTt} with a super-exponential bound on the connectivity. 
He asked whether a linear bound suffices.
Here we reduce the bound to a polynomial one.
Pokrovskiy~\cite{P} showed that a linear bound suffices to guarantee a linkage 
if we do not require it to be spanning.

\begin{theorem}\label{cor:spanningsubd}
Let $k\in \mathbb{N}$. Suppose that $T$ is a strongly $(k^2+3k)$-connected tournament and that $x_1,\dots,x_k,y_1,\dots,y_k$ are 
vertices in $T$ such that $x_i\neq y_i$ for all $i\in [k]$ and all the pairs $(x_i,y_i)$ are distinct.
Then $T$ contains pairwise internally disjoint paths $P_i$ from $x_i$ to $y_i$ such that
$\{x_1,\dots, x_k, y_1, \dots, y_k\}\cap V(P_i)=\{x_i, y_i\}$ and
$V(T) = \bigcup_{i=1}^{k} V(P_i)$. 
\end{theorem}

Both Theorem~\ref{thm:subd} and~\ref{cor:spanningsubd} can be deduced from Theorem~\ref{main theorem}
(but with weaker bounds). 
Instead, in Section~\ref{linkage} we adapt the argument from~\cite{CTt} to obtain a short direct proof of both these results.

\section{Notation and tools}\label{notation}

Given $k\in\mathbb{N}$, we let $[k]:=\{1,\dots,k\}$, $[k,k+\ell]:=\{k,\dots,k+\ell\}$ and $\log k:=\log_2 k$.
We write $V(G)$ and $E(G)$ for the set of vertices and the set of edges in a digraph $G$. We let $|G|:=|V(G)|$. If $u,v\in V(G)$ we write $uv$ for the
directed edge from $u$ to $v$.  We write  $d_G^-(v)$ and $d^+_G(v)$ for the in-degree and the out-degree of a vertex $v$ in $G$.
We write $\delta^-(G)$ and  $\delta^+(G)$ for the minimum in-degree and the minimum out-degree of $G$ and let $\delta^0(G):=\min\{\delta^-(G),\delta^+(G)\}$.
A set $A\subseteq V(G)$ \textit{in-dominates} a set $B\subseteq V(G)$ if for every vertex $b\in B$ there exists a vertex $a\in A$ such that $ba\in E(G)$. Similarly,
we say that $A$ \textit{out-dominates} $B$ if for every vertex $b\in B$ there exists a vertex $a\in A$ such that $ab\in E(G)$. 
We say that a tournament $T$ is \textit{transitive} if we may enumerate its vertices $v_1,\dots, v_m$ such that $v_iv_j\in E(T)$ if and only if $i<j$.
In this case we call $v_1$ the \textit{source} of $T$ and $v_m$ the \textit{sink} of $T$.
When referring to subpaths of tournaments, we always mean that
these paths are directed (i.e.~consistently oriented). The \textit{length} of a path is the number of its edges. We say that a path $P$
is \textit{odd} if its length is odd, and \textit{even} if its length is even. We say that two paths are disjoint if they
are vertex-disjoint. A tournament $T$ is \textit{strongly $k$-connected} if $|T|>k$ and for every set $F\subseteq V(T)$ with $|F|<k$ and
every ordered pair $x,y$ of vertices in $V(T)\setminus F$
there exists a path from $x$ to $y$ in $T-F$.
A tournament $T$ is called \textit{k-linked} if $|T|\geq 2k$ and whenever $x_1,\dots, x_k, y_1, \dots, y_k$ are $2k$ distinct vertices of~$T$
there exist disjoint paths $P_1,\dots, P_k$ such that $P_i$ is a directed path from $x_i$ to $y_i$ for each $i\in [k]$.

We now collect the tools which we need in our proof of Theorem~\ref{main theorem}. 
The following proposition
is a straightforward consequence of the definition of linkedness.

\begin{proposition}\label{linkedness alt def}
Let $k\in \mathbb{N}$. Then a tournament $T$ is $k$-linked if and only if $|T|\geq 2k$ and whenever $(x_1,y_1), \dots, (x_k,y_k)$ are ordered pairs
of (not necessarily distinct) vertices of $T$, there exist distinct internally disjoint paths $P_1, \dots, P_k$ such that for all $i\in [k]$
we have that $P_i$ is a directed path from $x_i$ to $y_i$ and that $\{x_1,\dots, x_k, y_1, \dots, y_k\}\cap V(P_i)=\{x_i, y_i\}$.
\end{proposition}

We will also use the following bound  from~\cite{P} on the connectivity which forces a tournament to be highly linked.

\begin{theorem}\label{connectivity implies linkedness}
For each $k\in \mathbb{N}$ every strongly $452 k$-connected tournament is $k$-linked.
\end{theorem}

The following two lemmas  guarantee that every tournament contains almost out-dominating and almost in-dominating sets which are not too large.
(A similar observation was also used in~\cite{KLOP}, see Lemmas 8.3 and 8.4.)

\begin{lemma}\label{out-dominating sets}
Let $T$ be a tournament, 
let $v\in V(T)$ and $c\in \mathbb{N}$ with $c \ge 2$. 
Suppose that $d^-_T(v) \ge 2^{c-1}$.
Then there exist disjoint sets $A, E\subseteq V(T)$ and a vertex $a \in A$ 
such that the following properties hold:
\begin{enumerate}[{\rm (i)}]
\item $2\leq |A|\leq c$ and $T[A]$ is a transitive tournament with source $a$ and sink $v$,
\item $A \setminus \{a \}$ out-dominates $V(T)\setminus (A\cup E)$,
\item $|E|\leq (1/2)^{c-2}d^-_T(v)$.
\end{enumerate}
\end{lemma}

\begin{proof}
Let $v_{1}:=v$. 
Roughly speaking, we will find $A$ by choosing vertices $v_1,\dots,v_i$ such that
the size of their common in-neighbourhood (i.e.\ the intersection of their individual in-neighbourhoods) is minimised at each
step. More precisely, suppose inductively 
that for some $1 \le i<c$ we have already found a set
$A_{i}=\{v_{1},\dots,v_{i}\}$ and a set $W_i$ such that the following holds:
\begin{itemize}
\item[(a)] $T[A_i]$ is a transitive tournament with sink $v_{1}$;
\item[(b)] $W_i=\emptyset$ or $W_i=\{a\}$ for some vertex $a$. Moreover, if $W_i=\{a\}$ then $E_i \cup A_i \subseteq N^+(a)$,
where $E_i:= \bigcap_{j=1}^i N^-(v_j) \setminus W_i$.
\item[(c)]  
$|E_{i}|\le\frac{1}{2^{i-1}}d^{-}(v).$ Moreover, $|E_i|>0$ if $W_i=\emptyset$. 
\end{itemize}
Note that (a)--(c) hold for $i=1$ if we let $A_{1}:=\{v_{1}\}$ and $W_1=\emptyset$. 

We first consider the case that $|E_{i}|\le\frac{1}{2^{c-2}}d^{-}(v)$.
If $W_i=\emptyset$, choose any vertex $a \in E_{i}$.
Else let $a$ be the vertex in $W_i$. 
In both cases let $A:= A_i \cup \{a\} $ and $E:=E_i\setminus \{a\} $.
Then $A$ and $E$ satisfy (i)--(iii).

So suppose next that $|E_{i}|>\frac{1}{2^{c-2}}d^{-}(v)$. 
(Note that in particular, this means that $|E_i| \ge 2$.)
By averaging, it follows that $E_{i}$ must contain a vertex $x$ of in-degree at most $|E_{i}|/2$ in
$T[E_{i}]$. If the in-degree of $x$ in $T[E_{i}]$ is nonzero or $W_i \neq \emptyset$, let $v_{i+1}:=x$.
Else let $v_{i+1}$ be a vertex of in-degree at most $|E_{i} \setminus \{x\}|/2$ in
$T[E_{i}\setminus \{ x \} ]$, and let $W_{i+1}:=\{x \}$ (note that we can find such a $v_{i+1}$ as $|E_i| \ge 2$).
Now let $A_{i+1}:=\{v_{1},\dots,v_{i+1}\}$ and
let $E_{i+1}:=(E_i \cap N^-(v_{i+1})) \setminus W_{i+1}  $.
Then $T[A_{i+1}]$ is a transitive tournament with sink $v_1$ and 
\[
|E_{i+1}|\le\frac{1}{2}|E_{i}|\le\frac{1}{2^{i}}d^{-}(v).
\]
So we have shown that (a)--(c) hold with $i+1$ playing the role of $i$.
By repeating this construction,  will eventually find $A$ and $E$ satisfying 
(i)--(iii). (Indeed, note that we must be in the first case for some $i <c$, in particular this implies that $|A| \le c$.)
\end{proof}

The next lemma follows immediately from Lemma \ref{out-dominating sets} by reversing the orientations of all edges.

\begin{lemma}\label{in-dominating sets}
Let $T$ be a tournament, let $v\in V(T)$ and $c\in \mathbb{N}$ with $c \ge 2$. 
Suppose that $d^+_T(v) \ge 2^{c-1}$.
Then there exist disjoint sets $B, E\subseteq V(T)$ and a vertex $b \in B$ such that the following properties hold:
\begin{enumerate}[{\rm (i)}]
\item $2\leq |B|\leq c$ and $T[B]$ is a transitive tournament with sink $b$ and source $v$,
\item $B \setminus \{b\}$ in-dominates $V(T)\setminus (B\cup E)$,
\item $|E|\leq (1/2)^{c-2}d^+_T(v)$.
\end{enumerate}
\end{lemma}

We will also need the following observation, which guarantees a small set $Z$ of vertices in a tournament such that every vertex
outside $Z$ has many out- and in-neighbours in $Z$.

\begin{proposition}\label{index selection}
Let $k,n \in \mathbb{N}$ and let $T$ be a tournament on $n\ge 4$ vertices. Then there is a set $Z\subseteq V(T)$ of size
$|Z|\le 3k\log n$ such that each vertex in $V(T)\setminus Z$
has at least $k$ out-neighbours and at least $k$ in-neighbours in $Z$. 
\end{proposition}
\begin{proof}
We may assume that $n\ge 3k\log n$. We will use the fact that every tournament on $n$ vertices contains an in-dominating set
of size at most $c:=\lceil \log n\rceil \le (3\log n)/2$.%
   \COMMENT{had $\lceil \log n\rceil+1$ before, but $\lceil \log n\rceil$ seems to be fine}
(This can be proved by choosing the vertices $x_1,x_2,\dots$ in the in-dominating set one by one, similarly as in the proof of Lemma~\ref{out-dominating sets}:
at the $i$th step we let $x_i$ be a vertex with the smallest out-degree in $T[\bigcap_{j<i} N^+(x_j)]$.) 
Choose an in-dominating set $V_1$ in $T$ of size at most $c$. Now consider the tournament $T- V_1$.
Choose an in-dominating set $V_2$ in $T- V_1$ with size at
most $c$. Continue in this way to obtain disjoint sets $V_1,\dots,V_k$. Proceed similarly to obtain
disjoint sets $U_1,\dots,U_k$, each of size at most~$c$, such that each $U_i$ is an out-dominating set in $T- (U_1\cup \dots \cup U_{i-1})$.
We can take $Z:= V_1\cup \dots \cup V_k\cup U_1 \dots \cup U_k$.
\end{proof}

\section{Proof of Theorem~\ref{main theorem}}\label{main section}

Let $X:=\{ x_1,x_2,\dots,x_{6k}\}\subseteq V(T)$ consist of $6k$ vertices whose in-degree in $T$ is as small as possible, and let
$Y:=\{ y_1,y_2,\dots, y_{6k}\}$ be a set of $6k$ vertices in $V(T)\setminus X$ whose out-degree in $T$ is as small as possible.
Define
$$
\hat{\delta}^-(T):= \min\limits_{v\in V(T)\setminus  X} d_T^-(v) \hspace{1cm} \textrm{and}\hspace{1cm}
\hat{\delta}^+(T):= \min\limits_{v\in V(T)\setminus Y} d_T^+(v).
$$
Let $c:= \left\lceil \log \left( 120k^2 \right) \right\rceil+2\le 9k$. Apply Lemmas~\ref{out-dominating sets}
and~\ref{in-dominating sets} with parameter $c$ repeatedly (removing the dominating sets each time) to obtain disjoint sets of vertices
$A_1,A_2,\dots, A_{6k},$ $B_1,B_2,\dots, B_{6k}$ and sets of vertices $E_{A_1},\dots, E_{B_{6k}}$
satisfying the following properties for all $i\in [6k]$, where we write
$D:=\bigcup_{i=1}^{6k}(A_i\cup B_i)$, 
\begin{enumerate}
\item[(D1)] $2\leq |A_i|\leq c$ and $T[A_i]$ is a transitive tournament with sink $x_i$ and source $a_i$,
\item[(D2)] $2\leq |B_i|\leq c$ and $T[B_i]$ is a transitive tournament with source $y_i$ and sink $b_i$,
\item[(D3)] $A_i \setminus \{a_i\} $ out-dominates $V(T)\setminus (D\cup E_{A_i})$ in~$T$,
\item[(D4)] $B_i \setminus \{b_i\}$ in-dominates $V(T)\setminus (D\cup E_{B_i})$ in~$T$,
\item[(D5)] $|E_{A_i}|\leq (1/2)^{c-2}\hat{\delta}^-(T)$,
\item[(D6)] $|E_{B_i}|\leq (1/2)^{c-2}\hat{\delta}^+(T)$.
\end{enumerate}
Let $$E_A:=\bigcup_{i \in [6k]} E_{A_i}, \ \ \ \ E_B:=\bigcup_{i \in [6k]} E_{B_i} \ \ \ \text{ and } \ \ \ E:= E_A\cup E_B.$$
Note that
\begin{equation}\label{E_A size}
 |E_A|\leq 6k \left( \frac{1}{2}\right)^{c-2}\hat{\delta}^-(T)\leq \frac{\hat{\delta}^-(T)}{20k} \ \ \ \ \text{and} \ \ \ \ 
 |E_B|\leq \frac{\hat{\delta}^+(T)}{20k}
\end{equation}
by our choice of $c$. Moreover, we may assume that $|E_A|\leq |E_B|$.
(The case $|E_A|>|E_B|$ follows by a symmetric argument.) In particular, this implies that
\begin{equation}\label{E size}
|E|\leq |E_A|+|E_B|\leq 2|E_B|\leq \frac{\hat{\delta}^+(T)}{10k}.
\end{equation}

We will iteratively colour the vertices of $T$ with colours $\alpha$ and $\beta$, and at each step $V_\alpha$ will
consist of all vertices of colour $\alpha$ and $V_\beta$ is defined similarly. At the end of our argument, every vertex of $T$ will
be coloured either with $\alpha$ or with $\beta$, i.e.~$V_\alpha,V_\beta$ will form a partition of $V(T)$.
Our aim is to colour the vertices in such a way that we can take $V_1:=V_\alpha$ and $V_2:=V_\beta$. 

We say a path $P$ is \textit{alternating} if the colour of the vertices on $P$ alternates as we move along~$P$.
$P$ is \textit{monochromatic} if all vertices of $P$ have the same colour.

At each step and for each $\gamma \in\{\alpha,\beta\}$, we call a vertex $v\in V_\gamma$ \textit{forwards-safe}
if for any set $F\not\ni v$ of at most $k-1$ vertices, there is a directed monochromatic path (possibly of length $0$) in $T[V_\gamma\setminus F]$ from $v$
to $V(T)\setminus (D\cup E_B\cup F)$. 
Similarly, we say that $v\in V_\gamma$ is \textit{backwards-safe}
if for any set $F\not\ni v$ of at most $k-1$ vertices, there is a directed monochromatic path (possibly of length $0$) in
$T[V_\gamma\setminus F]$ from $V(T)\setminus (D\cup E_A\cup F)$ to $v$. 

We call a vertex $v\in V_\gamma$ \textit{alternating-forwards-safe}
if for any set $F\not\ni v$ of at most $k-1$ vertices, there is a directed alternating path (possibly of length $0$)
in $T-F$ from $v$ to $V(T)\setminus (D\cup E_B\cup F)$. 
Similarly, we say that $v\in V_\gamma$ is \textit{alternating-backwards-safe}
if for any set $F\not\ni v$ of at most $k-1$ vertices, there is a directed alternating path (possibly of length $0$)
in $T-F$ from $V(T)\setminus (D\cup E_A\cup F)$ to $v$. 

We say that a vertex $v$ is \textit{safe} if it is safe in all four respects.

Note that the following properties are satisfied at every step (for each $\{\gamma,\delta\} =\{\alpha,\beta\}$):
\begin{itemize}
\item[(S1)] all coloured vertices in $V(T)\setminus (D\cup E)$ are safe,
\item[(S2)] all coloured vertices in $V(T) \setminus (D\cup E_B)$ are forwards-safe as well as alternating-forwards-safe and all coloured vertices in
$V(T) \setminus (D\cup E_A)$ are backwards-safe as well as alternating-backwards-safe,
\item[(S3)] if $v\in V_\gamma$ has at least $k$ forwards-safe out-neighbours of colour $\gamma$ then $v$ itself is forwards-safe,
the analogue holds if $v$ has at least $k$ backwards-safe in-neighbours of colour $\gamma$,
\item[(S4)] if $v\in V_\gamma$ has at least $k$ alternating-forwards-safe out-neighbours of colour $\delta$ with $\delta\neq \gamma$ then $v$ itself is alternating-forwards-safe,
the analogue holds if $v$ has at least $k$ alternating-backwards-safe in-neighbours of colour $\delta$,
\item[(S5)] if $v\in V_\gamma$ is safe and in the next step we colour some more (previously uncoloured) vertices then $v$ is still safe.
\end{itemize}

In what follows, by a (partial) colouring of the vertices of $T$ we always mean a colouring with colours $\alpha$ and $\beta$
in which all the vertices in 
\begin{align*}
D_1: & = \bigcup_{i\in [k]} (A_i \cup B_i) \cup \bigcup_{i\in [3k+1, 5k]} (A_i \setminus\{a_i\}) \cup  \bigcup_{i\in [3k+1, 4k]\cup [5k+1,6k]} (B_i\setminus\{b_i\})\\
&\cup \{a_i \mid i\in [2k+1, 3k]\cup [5k+1,6k]\}\cup \{b_i \mid i\in [2k+1, 3k]\cup [4k+1,5k]\}
\end{align*}
are coloured $\alpha$, and all the vertices in $D_2:=D\setminus D_1$ are coloured $\beta$.

\medskip

\noindent
{\bf Claim 0:} \textit{Suppose that there are paths $P_{1}, \dots,P_{6k}$ of $T$ satisfying the following properties:
\begin{itemize}
\item for each $i\in [6k]$ the path $P_i$ joins $b_i$ to $a_i$,
\item the paths $P_1,\dots,P_{6k}$ are disjoint from each other and
meet $D$ only in their endvertices.
\end{itemize}
Suppose that we have coloured all vertices of $T$ such that
\begin{itemize}
\item every vertex in $D_1\cup V(P_1)\cup \dots \cup V(P_k)$ is coloured $\alpha$,
\item every vertex in $D_2\cup V(P_{k+1})\cup \dots \cup V(P_{2k})$ is coloured $\beta$,
\item $P_{2k+1},\dots, P_{6k}$ are alternating,
\item every vertex is safe.
\end{itemize}
Then the sets $V_1:=V_\alpha$ and $V_2:=V_\beta$ form a partition of $V(T)$ as required in Theorem~\ref{main theorem}.}

\smallskip

Note that the conditions of Claim~0 imply that $P_i$ must be an even path for $i\in [2k+1,4k]$ and an odd path for $i\in [4k+1,6k]$.

To prove Claim~0, we first show that $T[V_\alpha]$ is strongly $k$-connected. So consider any set $F$ of at most $k-1$ vertices
and any two vertices $x,y\in V_\alpha\setminus F$.
We need to check that $T[V_\alpha\setminus F]$ contains a path from
$x$ to $y$. Since $x$ is forwards-safe there exists a path $Q_x$ in $T[V_\alpha\setminus F]$ from $x$
to some vertex $x'\in V_\alpha\setminus (D\cup E_B\cup F)$. Similarly, since $y$ is backwards-safe there exists a
path $Q_y$ in $T[V_\alpha\setminus F]$ 
from some vertex $y'\in V_\alpha\setminus ( D\cup E_A\cup F)$ to $y$.
Let $i\in [k]$ be such that $F$ avoids $A_{i}\cup V(P_i)\cup B_{i}$.
Since $x'\notin D\cup E_B$, (D4) implies
that $x'$ sends an edge to $B_{i}$. Similarly, since $y'\notin D\cup E_A$, (D3) implies
that $y'$ receives an edge from $A_{i}$. Altogether this implies that
$T[V(Q_x)\cup V(Q_y)\cup A_{i}\cup V(P_i)\cup B_{i}]\subseteq T[V_\alpha\setminus F]$
contains path from $x$ to $y$, as desired.

A similar argument shows that $V_\beta$ is strongly $k$-connected too.
It remains to show that $T[V_\alpha,V_\beta]$ is stongly $k$-connected. Consider any set $F$ of at most $k-1$ vertices and any two vertices
$x,y \in V(T)\setminus F$.
We will show that there is an alternating path between $x$ and $y$ avoiding $F$.
Since $x$ is alternating-forwards-safe there exists an alternating path $Q_x$ in $T[V_\alpha, V_\beta]-F$ from $x$
to some vertex $x'\in V(T) \setminus (D\cup E_B\cup F)$. Similarly, since $y$ is backwards-safe there exists a
path $Q_y$ in $T[V_\alpha,V_\beta]-F$ 
from some vertex $y'\in V[T]\setminus ( D\cup E_A\cup F)$ to $y$. We now choose an index $i$ as follows:
\begin{itemize}
\item If $x',y' \in V_\alpha$, let $i\in [2k+1,3k]$ be such that $F$ avoids $A_i\cup V(P_i) \cup B_i$.
\item If $x',y' \in V_\beta$, let $i\in [3k+1,4k]$ be such that $F$ avoids $A_i\cup V(P_i) \cup B_i$.
\item If $x' \in V_\alpha$ and $y'\in V_\beta$, let $i\in [4k+1,5k]$ be such that $F$ avoids $A_i\cup V(P_i) \cup B_i$.
\item If $x' \in V_\beta$ and $y'\in V_\alpha$, let $i\in [5k+1,6k]$ be such that $F$ avoids $A_i\cup V(P_i) \cup B_i$.
\end{itemize}
Since $x'\notin D\cup E_B$, (D4) implies that $x'$ sends an edge to $B_{i}\setminus\{b_i\}$. Similarly, since $y'\notin D\cup E_A$, (D3) implies
that $y'$ receives an edge from $A_{i}\setminus\{a_i\}$. Altogether this implies that
$T[V(Q_x)\cup V(Q_y)\cup A_{i}\cup V(P_i)\cup B_{i}]\subseteq T-F$ contains an alternating path from $x$ to $y$, as desired.
This completes the proof of Claim~0.

\medskip

\noindent
{\bf Claim 1:} \textit{Consider a partial colouring of $V(T)$ and let $C$ denote the set of previously coloured vertices.
(So $D\subseteq C$.) Let $Z\subseteq V(T)\setminus (X\cup Y)$ and $N\subseteq V(T)\setminus Z$
and suppose that $9k^2|Z|+|C\cup N|\le 5\cdot 10^8 k^6\log(2k)$. Then for every colouring of the vertices in $Z\setminus C$ there
is a set $Z'\subseteq V(T)\setminus (Z\cup N\cup C)$ and a colouring of the vertices in $Z'$ such that every vertex in $Z\cup Z'$
is safe and $|Z\cup Z'|\le 9k^2|Z|$.}

\smallskip

To prove Claim~1, note that the strong $10^9k^6\log(2k)$-connectivity of $T$ implies that $\delta^0(T)\geq 10^9k^6\log(2k)$.
Hence%
   \COMMENT{We need a bound on $\hat{\delta}^+(T)-5k|E|$ (instead of $\hat{\delta}^+(T)-|E|$) later on in the proof of Claim~4.}
\begin{equation}\label{referee equation E_A}
\hat{\delta}^-(T)-5k|E_A| \stackrel{(\ref{E_A size})}{\ge } \frac{\hat{\delta}^-(T)}{2}\ge \frac{\delta^0(T)}{2}
\ge  5\cdot 10^8 k^6\log(2k),
\end{equation}
and similarly
\begin{equation}\label{referee equation E}
\hat{\delta}^+(T)-5k|E| \stackrel{(\ref{E size})}{\ge } \frac{\hat{\delta}^+(T)}{2}\ge  5\cdot 10^8k^6\log(2k).
\end{equation}

Consider any colouring of $Z\setminus C$. 
For each vertex $z\in Z$  in turn we greedily choose $2k$ uncoloured in-neighbours outside $N\cup E_A$,
and colour $k$ of them $\alpha$ and the remaining $k$ by $\beta$. (We do not modify $C$ in this process.) To see that
we can choose all these vertices to be distinct from each other, note that the total number of vertices we wish to choose is $2k|Z|$ and
\begin{equation*}
|C\cup N\cup Z|+2k|Z|\le 5\cdot 10^8 k^6\log(2k) \stackrel{(\ref{referee equation E_A})}{\le} \hat{\delta}^-(T)-|E_A|.
\end{equation*}
For each vertex in $Z$ as well as for each of the $2k|Z|$ vertices that we coloured in the previous step in turn, we greedily choose $2k$
uncoloured out-neighbours outside $N \cup E$, and colour $k$ of them by $\alpha$ and the remaining $k$ by $\beta$. To see that
we can choose all these vertices to be distinct from each other, note that the total number of vertices we wish to choose is $2k(1+2k)|Z|$ and
$$
|C\cup N\cup Z|+2k|Z|+2k(1+2k)|Z|\le |C\cup N|+9k^2|Z|\le 5\cdot 10^8k^6\log(2k)\stackrel{(\ref{referee equation E})}{\le} \hat{\delta}^-(T)-|E|.
$$
Let $Z'$ be the set of vertices outside $C\cup Z$ that we coloured. Then $Z'\cap N=\emptyset$.
Moreover, using (S1)--(S4) it is easy to check that every vertex in $Z\cup Z'$ is safe.
This completes the proof of Claim~1.

\medskip

Recall that we have already coloured all the vertices in $D_1$ by $\alpha$ and all the vertices in $D_2$ by $\beta$.
Step by step, we will now colour further vertices of $T$. Our final aim is to arrive at a colouring of $V(T)$ which is as described
in Claim~0. The first step is to colour some more vertices in order to achieve that all the coloured vertices are safe.
In what follows, when saying that we colour some additional vertices we always mean that these vertices are uncoloured so far.

\medskip

\noindent {\bf Claim 2:} \textit{We can colour some additional vertices of $T$ in such a way that every coloured vertex is safe
and the set $C_1$ consisting of all vertices coloured so far satisfies} $|C_1|\le 1500k^4$.

\smallskip

To prove Claim~2, for every $v\in \{x_1,\dots, x_{6k},y_1,\dots,y_{6k}\}$ in turn, we greedily choose $2k$ uncoloured in-neighbours
and $2k$ uncoloured out-neighbours, all distinct from each other, and colour $k$ of the in-neighbours and $k$ of the out-neighbours by $\alpha$ and the remaining $2k$ in/out-neighbours by $\beta$.

Let $Z^*$ denote the set of $4k\cdot 12k= 48k^2$ new vertices we just coloured and let $Z:=Z^*\cup (D\setminus (X\cup Y))$.
Then $|Z|\le |Z^*|+|D|\le 48k^2+ c\cdot 12k\le 156k^2$. Apply Claim~1 with $N:=\emptyset$ to find a set $Z'$ of uncoloured vertices and a colouring of these vertices such that
all the vertices in $Z\cup Z'$ are safe and $|Z\cup Z'|\le 9k^2\cdot |Z|\le 1500 k^4$. Our choice of $Z^*$ and (S3), (S4) together now imply
that the vertices in $X\cup Y$ are safe as well. This completes the proof of Claim~2.

\medskip

Suppose that $P$ is a path whose endvertices are already coloured, but whose internal vertices are still uncoloured.
We say that we colour (the internal vertices of) $P$ in an \textit{alternating manner consistent with its endvertices} if the
colouring results in an alternating path. (So for example, if the endvertices of $P$ have the same colour, then $P$ needs to be
an even path.)

\medskip

\noindent {\bf Claim 3:} \textit{There are paths $P_1,P_2,\dots, P_{6k}$ of $T$ satisfying the following properties:
\begin{enumerate}[{\rm (i)}]
\item for each $i\in [6k]$, the path $P_i$ joins $b_i$ to $a_i$,
\item the paths $P_1,\dots, P_{6k}$ are disjoint from each other and meet $C_1$ only in their endvertices,
\item we can colour the internal vertices of $P_1,\dots, P_k$ by $\alpha$, the internal vertices of $P_{k+1},\dots, P_{2k}$ by $\beta$ and
the internal vertices of $P_{2k+1},\dots, P_{6k}$ in an alternating manner consistent with their endvertices and can colour some additional vertices
such that the set $C_4$ of all coloured vertices satisfies the following properties:
\begin{itemize}
\item[{\rm (a)}] all vertices in $C_4$ are safe,
\item[{\rm (b)}] there is a set $C^0\subseteq C_4$ such that the number of coloured vertices outside $C^0$ is at most $3\cdot 10^7 k^6\log(2k)$,
\item[{\rm (c)}] every vertex outside $C_4$ which has an in-neighbour in
$C^0$ has at least $k$ in-neighbours of each colour, 
and every vertex outside $C_4$ which has an out-neighbour in $C^0$
has at least $k$ out-neighbours of each colour.
\end{itemize}
\end{enumerate}}

We will prove Claim~3 via a sequence of subclaims. 
For $i\in [6k]$ we define an $i$-\textit{path}
to be a directed path from the sink $b_i$ of $B_i$ to the source $a_i$ of $A_i$ whose internal vertices lie outside~$C_1$.
Ideally, we would like to find disjoint $i$-paths $P_i$ (one for each $i\in [6k]$) such that the following properties hold:
\begin{enumerate}
\item we can colour all the internal vertices of $P_1,\dots, P_k$ by $\alpha$, the internal vertices of
$P_{k+1},\dots, P_{2k}$ by $\beta$ and
the internal vertices of $P_{2k+1},\dots, P_{6k}$ in an alternating manner consistent with their endvertices,
\item by colouring some additional vertices we can achieve that all coloured vertices are safe.
\end{enumerate}
For each $i\in  [6k]$ we will first try to find a short $i$-path $P_i$ such that all these
$i$-paths are disjoint and such that for each $i\in [2k+1,6k]$ the length of the path $P_i$ has the correct parity in order
to ensure that the internal vertices of $P_i$ can be coloured in an alternating manner consistent with the endvertices
of $P_i$ (so $P_i$ needs to be even for $i\in [2k+1,4k]$ and odd for $i\in [4k+1,6k]$).
We will then colour the vertices on these short $i$-paths as well as
some additional vertices such that (1) and (2) are satisfied for the set $I_{short}$
of all indices $i$ for which we have been able to choose a short $i$-path (see Claim~3.1). This provides some of the paths required in Claim~3.
To find the remaining paths, for all $i\notin I_{short}$ we will choose $10^5 k^4\log(2k)$ $i$-paths $Q_{i,1},\dots,Q_{i,10^5 k^4\log(2k)}$
such that all these paths are internally disjoint from each other. For each $i\notin I_{short}$ with $i\in [2k]$
there will be three distinct indices $j_{i,1},j_{i,2},j_{i,3}\in [10^5 k^4\log(2k)]$ such that the path $P_i$ required
in Claim~3 will consist of an initial segment of $Q_{i,j_{i,1}}$, a middle segment of $Q_{i,j_{i,2}}$, a final
segment of $Q_{i,j_{i,3}}$ as well as two edges joining these three segments. Similarly, for each $i\notin I_{short}$ with $i\in [2k+1,6k]$
the path $P_i$ required in Claim~3 will either be one of the $Q_{i,j}$ or will consist of an initial segment of
$Q_{i,j_{i,1}}$ and a final segment of $Q_{i,j_{i,2}}$ as well as an edge joining these two segments. 

We will now choose the short $i$-paths.
Let $\mathcal{P}_{short}^{correct}$ be a collection of $i$-paths satisfying the following properties:
\begin{itemize}
\item[(P1)] for each $i\in [6k]$, $\mathcal{P}_{short}^{correct}$ contains at most one $i$-path,
\item[(P2)] all the paths in $\mathcal{P}_{short}^{correct}$ are disjoint from each other,
\item[(P3)] each path has length at most $10k+10$,
\item[(P4)] for each $i\in [2k+1,6k]$ for which $\mathcal{P}_{short}^{correct}$ contains an $i$-path, this path $P_i$ has the correct parity,
meaning that $P_i$ is even if $i\in [2k+1,4k]$ and odd if $i\in [4k+1,6k]$,
\item[(P5)] subject to the above conditions, $|\mathcal{P}_{short}^{correct}|$ is as large as possible.
\end{itemize}
Let $I_{short}^{correct}$ be the set of all those indices $i\in [6k]$ for which $\mathcal{P}_{short}^{correct}$
contains an $i$-path, and let $P_i$ denote this $i$-path. 
Let $V_{short}^{correct}$ be the set of all internal vertices of the $P_i$ for all $i \in I_{short}^{correct}$.
Moreover, set $I_{long}:= [6k]\setminus I_{short}^{correct}$.
Recall that the definition of an $i$-path implies that all the vertices in $V_{short}^{correct}$ are uncoloured so far (i.e.~$V_{short}^{correct}\cap C_1=\emptyset$).

\medskip

\noindent{\bf Claim 3.1:} \textit{We may colour all vertices in} $V_{short}^{correct}$ \textit{as well as some additional vertices of} $T$
\textit{such that the following properties hold:
\begin{enumerate}[{\rm (i)}]
\item for each $i\in I_{short}^{correct}$, all the vertices on $P_i$ are coloured $\alpha$ if $i\in [k]$ and $\beta$ if $i\in [k+1,2k]$,
\item for each $i\in I_{short}^{correct}\setminus [2k]$, $P_i$ is coloured in an alternating manner consistent with its endvertices,
\item the set $C_2$ consisting of all vertices coloured so far has size $|C_2|\le 4000k^4$ and all vertices in $C_2$ are safe,
\item for each $i\in I_{long}$, any $i$-path whose internal vertices lie in $V(T)\setminus C_2$ is either $b_ia_i$ or has length at least $10k+10$.
\end{enumerate}}

\smallskip

To prove Claim~3.1, consider any $i\in I_{short}^{correct}$ and colour all internal vertices of $P_i$ by $\alpha$ if $i\in [k]$, by $\beta$ if $i\in [k+1,2k]$,
and in an alternating manner consistent with the endvertices of $P_i$ if $i\in [2k+1,6k]$ (this is possible by (P4)). Note that $|V_{short}^{correct}|\le 6k(10k+9)\le 120k^2$.
Together with Claim~1 (applied with $N:=\emptyset$ and $Z:=V_{short}^{correct}$)
and Claim~2 this implies Claim~3.1(i)--(iii), with room to spare in (iii). Indeed, the set $C'_2$ of vertices coloured so far has size $|C'_2|\le 3000k^4$.%
   \COMMENT{$|C'_2|\le |C_1|+9k^2\cdot 120k^2\le 1500k^4+1200k^4$}

We will now colour some additional vertices to ensure that (iv) holds too.
Consider any $i\in I_{long}$. If there exists an $i$-path $P$ whose internal vertices lie in $V(T)\setminus C'_2$ and whose length is at most $10k+9$, then $P$ must have incorrect parity,
i.e.~$P$ is odd if $i\in [2k+1, 4k]$ and even if $i\in [4k+1, 6k]$. Note that there cannot be two such $i$-paths of length at least two which are internally disjoint from each other.
Indeed, if $P=v_1\dots v_a$ and $P'=v'_1 \dots v'_{a'}$ are two such $i$-paths which are internally disjoint, then we may assume that
$v_2v'_2\in E(T)$ and so $v_1v_2v'_2v'_3\dots v'_{a'}$ is an $i$-path of length at most $10k+10$
with the correct parity which is disjoint from all the other paths in $\mathcal{P}_{short}^{correct}$, a contradiction to~(P5).

Let $\mathcal{P}_{short}^{incorrect}$ be a collection of $i$-paths whose internal vertices lie in $V(T)\setminus C'_2$ and whose length is at least two and at most $10k+9$,
such that all these paths are disjoint from each other and, subject to these properties, such that $|\mathcal{P}_{short}^{incorrect}|$ is as large as possible.
Let $V_{short}^{incorrect}$ be the set of all internal vertices on these paths. Thus $|V_{short}^{incorrect}|\le 4k\cdot (10k+8)\le 100k^2$.
Colour all vertices in $V_{short}^{incorrect}$ with $\alpha$ and apply Claim~1 again (with $N:=\emptyset$ and $Z:=V_{short}^{incorrect}$).
Then the set $C_2$ of all vertices coloured so far satisfies
$|C_2|\le 3000k^2+9k^2\cdot 100k^2\le 4000k^4$, so (iii) still holds. Moreover, now (iv) holds too. This completes the proof of Claim~3.1.

\medskip

Claim~3.1(iii) implies that all uncoloured vertices together with
the $a_i$ and $b_i$ for all $i \in I_{long}$ induce a strongly $(7 \cdot 452 \cdot 10^5 k^5\log(2k))$-connected subtournament $T'$ of~$T$ (with some room to spare).
Theorem \ref{connectivity implies linkedness} implies that $T'$ is $7\cdot 10^5 k^5\log(2k)$-linked.
Together with Proposition~\ref{linkedness alt def} this implies that for each $i \in I_{long}$ we can find $10^5 k^4\log(2k)$
$i$-paths in $T'$ such that all these $10^5 k^4\log(2k) |I_{long}|$ paths have length at least two and are internally disjoint from each other and
such that the internal vertices on all these paths
lie outside~$C_2$. We choose this collection of $10^5 k^4\log(2k) |I_{long}|$ paths such that the set $V_{long}$ of all internal vertices on these paths
is as small as possible. For all $i\in I_{long}$ and all $j\in [10^5 k^4\log(2k)]$, let $Q_{i,j}$ denote the $j$th $i$-path we chose.
Write $Q_{i,j} = q_{i,j}^0q_{i,j}^1\dots q_{i,j}^{|Q_{i,j}|}$, so that $q_{i,j}^0$ is $b_i$ and $q_{i,j}^{|Q_{i,j}|}$ is $a_i$.
Claim~3.1(iv) implies that each $Q_{i,j}$ must have length at least $10k+10$. Moreover, the minimality of $|V_{long}|$ implies the following:
\begin{enumerate}
\item[(Q1)] the interior of each $Q_{i,j}$ induces a backwards-transitive path,
\item[(Q2)] if $v\in V(T)\setminus (C_2\cup V_{long})$ is an out-neighbour of $q_{i,j}^s$, then $v$ is also an out-neighbour of
$q_{i,j}^{s'}$ for all $s' \geq s+3$,
\item[(Q3)] if $v\in V(T)\setminus (C_2\cup V_{long})$ is an in-neighbour of $q_{i,j}^s$, then $v$ is also an in-neighbour of
$q_{i,j}^{s'}$ for all $s' \leq s-3$.
\end{enumerate}
Let ${\rm int}(Q_{i,j}):=q_{i,j}^1\dots q_{i,j}^{|Q_{i,j}|-1}$ denote the interior of $Q_{i,j}$.
Let $Q_{i,j}^1, \dots, Q_{i,j}^9$ be disjoint segments of ${\rm int}(Q_{i,j})$ such that ${\rm int}(Q_{i,j}) = Q_{i,j}^1\dots Q_{i,j}^9$,
$|Q_{i,j}^1|=|Q_{i,j}^2|=|Q_{i,j}^8|=|Q_{i,j}^9|=k$, $|Q_{i,j}^3|=|Q_{i,j}^7|=k+2$ and $|Q_{i,j}^4|=|Q_{i,j}^6|=2k+2$.
We let $$Q_{i,j}^0 := Q_{i,j}^1\cup Q_{i,j}^2\cup Q_{i,j}^3  \cup Q_{i,j}^7\cup Q_{i,j}^8\cup Q_{i,j}^9$$
and write 
$$V^0_{long}:=\bigcup_{(i,j)\in I_{long}\times [10^5 k^4 \log(2k)]} V(Q^0_{i,j}) \ \ \ \text{ and } \ \ \
\overline{V}_{long}:=\bigcup_{(i,j)\in I_{long}\times [10^5 k^4 \log(2k)]} V(Q^{0}_{i,j}\cup Q^{4}_{i,j} \cup Q^{6}_{i,j}).
$$ Thus $V^0_{long}\subseteq \overline{V}_{long}\subseteq V_{long}$ and
$$|V^0_{long}|\le |\overline{V}_{long}|\le (10k+8)\cdot 6k\cdot 10^5 k^4 \log(2k) \le 2\cdot 10^7 k^6\log(2k).$$

\noindent{\bf Claim 3.2:} \textit{There exist disjoint index sets $I_{R,\alpha}, I_{R,\beta} \subseteq I_{long}\times [10^5 k^4 \log(2k)]$
such that, writing
$$R_\alpha := \bigcup_{(i,j)\in I_{R,\alpha}} V(Q_{i,j}^{0}) \ \ \ \text{and} \ \ \ R_\beta :=\bigcup_{(i,j)\in I_{R,\beta}} V(Q_{i,j}^{0}), $$ 
for each $(i,j) \in I_{long}\times [10^5 k^4 \log(2k)]$  every vertex in $V(Q_{i,j}^{0})\setminus (R_\alpha\cup R_\beta)$ has at least $k$
in-neighbours and at least $k$ out-neighbours in each of $R_\alpha$ and $R_\beta$.
Also $|I_{R,\alpha}|,|I_{R,\beta}|\le 100 k\log(2k)$ and $|R_\alpha|,|R_\beta| \leq  1000 k^2\log(2k)$.}

\smallskip 

To prove Claim~3.2, apply Proposition~\ref{index selection} to $T[V^0_{long}]$ to find a
set $Z_\alpha \subseteq V^0_{long}$ with $|Z_\alpha| \le 3k \log |V^0_{long}| \le 100k \log(2k)$ and%
  \COMMENT{$3k \log |V^0_{long}|\le 3k\log (2\cdot 10^7 k^6\log(2k))\le 3k(1+7\log 10+6\log k+\log\log(2k))
  \le 3k\cdot 33\log(2k)$}
such that every vertex in $V^0_{long} \setminus Z_\alpha$ has at least $k$ out-neighbours and $k$ in-neighbours in $Z_\alpha$.
Let $I_{R,\alpha} := \{(i,j): V(Q_{i,j}^{0}) \cap Z_\alpha \neq \emptyset\}$ and $I':=(I_{long}\times[10^5 k^4 \log(2k)]) \setminus I_{R,\alpha}$.
We now consider $W := \bigcup_{(i,j) \in I'} V(Q_{i,j}^{0})$.
By Proposition~\ref{index selection} applied to $T[W]$, 
there exists a set $Z_\beta \subseteq W$ with $|Z_\beta| \le 3k \log |W| \le 100k \log(2k)$ and such that every vertex in
$W \setminus Z_\beta$ has at least $k$ out-neighbours and in-neighbours in $Z_\beta$. Let
$I_{R,\beta} := \{(i,j)\in I': V(Q_{i,j}^{0}) \cap Z_\beta \neq \emptyset\}$.

Let $R_\alpha$ and $R_\beta$ be as defined in the statement of Claim~3.2.
Then by definition of $I_{R,\alpha}$ and $I_{R,\beta}$,
for each $(i,j)\in I_{long}\times [10^5 k^4 \log(2k)]$
every vertex in $V(Q_{i,j}^{0})\setminus (R_\alpha\cup R_\beta)$ has at least
$k$ in-neighbours and at least $k$ out-neighbours in each of $R_\alpha$ and $R_\beta$.
Also $|R_\alpha|,|R_\beta| \leq  (6k+4) \cdot 100k\log(2k) \le 1000k^2 \log(2k)$.
This completes the proof of Claim 3.2.

\medskip

Let $I_{R}:= I_{R,\alpha}\cup I_{R,\beta}$, $R:=R_\alpha\cup R_\beta$
and $$R^{4,6} := \bigcup_{(i,j)\in I_R} V(Q_{i,j}^4 \cup Q_{i,j}^6).$$

\medskip

\noindent {\bf Claim 3.3:} \textit{We may colour all vertices in $R_\alpha\cup R_\beta\cup R^{4,6}$ as well as some additional
vertices lying outside $\overline{V}_{long}$ such that
\begin{enumerate}[{\rm (i)}]
\item all vertices in $R_\alpha$ are coloured $\alpha$, all vertices in $R_\beta$ are coloured $\beta$,
\item for each $(i,j)\in I_{R}$ and each $s\in \{4,6\}$, $Q_{i,j}^s$ is an alternating path,
\item all coloured vertices are safe,
\item the set $C_3$ consisting of all vertices coloured so far has size $|C_3|\le  4\cdot 10^4 k^4\log(2k)$.
\end{enumerate}}

To prove Claim~3.3, colour the vertices in $R_\alpha\cup R_\beta\cup R^{4,6}$ such that (i) and (ii) hold.
Apply Claim~1 with $C_2$, $R_\alpha\cup R_\beta\cup R^{4,6}$, $\overline{V}_{long}\setminus (R_\alpha\cup R_\beta\cup R^{4,6})$
playing the roles of $C$, $Z$, $N$ to obtain a set $Z'\subseteq V(T)\setminus (\overline{V}_{long}\cup C_2)$
and a colouring of the vertices in $Z'$ such that every vertex in $R_\alpha\cup R_\beta\cup R^{4,6} \cup Z'$ is safe and%
   \COMMENT{The calculation below also shows that Claim~1 can be applied, indeed
$9k^2|Z|+|C_2\cup \overline{V}_{long}|\le 4\cdot 10^4 k^4\log(2k)+4000k^4+|\overline{V}_{long}|\le 4\cdot 10^4 k^4\log(2k)+4000k^4+2\cdot 10^7k^6\log(2k)$}  
\begin{align*}
|C_3| & \leq |C_2|+|R_\alpha\cup R_\beta\cup R^{4,6} \cup Z'|\le 4000k^4+ 9k^2 \cdot  (2\cdot 1000k^2\log(2k)+ (4k+4)\cdot 200k\log(2k))\\
& \le 4\cdot 10^4 k^4\log(2k).
\end{align*}
This completes the proof of Claim~3.3.

\medskip 

\noindent {\bf Claim 3.4:} \textit{For each $i\in I_{long}$ there is an $i$-path $P_i$ such that the following properties hold:}
\begin{enumerate}[{\rm (i)}]
\item \textit{$P_i$ has no internal vertices in $C_3$, and $P_i$ and $P_{i'}$ are disjoint whenever $i\neq i'$,}

\item \textit{if $i\in I_{long}\cap [2k]$, then there exists three distinct indices $j_{i,1}, j_{i,2}, j_{i,3}\in [10^5 k^4 \log(2k)]$
such that $P_i = b_iQ_{i,j_{i,1}}^1Q_{i,j_{i,1}}^2 q_{i,j_{i,1}}^{2k+1}
Q_{i,j_{i,2}}^3\dots Q_{i,j_{i,2}}^7 q_{i,j_{i,3}}^{|Q_{i,j_{i,3}}|-2k-1}Q_{i,j_{i,3}}^8Q_{i,j_{i,3}}^9a_i$,}

\item \textit{if $i\in I_{long}\cap [2k+1,6k]$, then either $P_i=Q_{i,j_i}$ for some $j_i\in [10^5 k^4 \log(2k)]$ or
there exist distinct $j_{i,1}, j_{i,2}\in [10^5 k^4 \log(2k)]$ such that 
$P_i = b_iQ_{i,j_{i,1}}^1\dots Q_{i,j_{i,1}}^4 q_{i,j_{i,1}}^{5k+5}Q_{i,j_{i,2}}^5\dots Q_{i,j_{i,2}}^9a_i$,}

\item \textit{$P_i$ is even if $i\in I_{long}\cap [2k+1,4k]$ and odd if $i\in I_{long}\cap [4k+1,6k]$.}
\end{enumerate}

(Recall that $q_{i,j_{i,1}}^{2k+1}$ is the first vertex of $Q_{i,j_{i,1}}^3$, $q_{i,j_{i,3}}^{|Q_{i,j_{i,3}}|-2k-1}$ is the last vertex of
$Q_{i,j_{i,3}}^7$ and $q_{i,j_{i,1}}^{5k+5}$ is the first vertex of $Q_{i,j_{i,1}}^5$.)
To prove Claim~3.4, note that since $|C_3|\leq 4\cdot 10^4 k^4\log(2k) < 10^5k^4\log(2k) - 5$, for each $i\in I_{long}$ there are at
least five paths $Q_{i,s_{i,1}}, Q_{i,s_{i,2}}, Q_{i,s_{i,3}}, Q_{i,s_{i,4}}, Q_{i,s_{i,5}}$ whose internal vertices avoid~$C_3$. 

Suppose first that $i\in  I_{long}\cap [2k]$. Consider the subtournament $T_i$ of $T$ spanned by $q_{i,s_{i,t}}^{|Q_{i,s_{i,t}}|-2k-1}$
for $t=1,2,3,4,5$. $T_i$ contains at least two vertices of out-degree at least two, assume they are 
$q_{i,s_{i,1}}^{|Q_{i,s_{i,1}}|-2k-1}, q_{i,s_{i,2}}^{|Q_{i,s_{i,2}}|-2k-1}$. We may also assume that $q_{i,s_{i,1}}^{2k+1}$ sends
an edge to $q_{i,s_{i,2}}^{2k+1}$. Finally, since $q_{i,s_{i,2}}^{|Q_{i,s_{i,2}}|-2k-1}$ has at least two
outneighbours in $T_i$, we may assume $q_{i,s_{i,2}}^{|Q_{i,s_{i,2}}|-2k-1}$ sends an edge to $q_{i,s_{i,3}}^{|Q_{i,s_{i,3}}|-2k-1}$.
We set $j_{i,t} := s_{i,t}$ and let $P_i$ be as described in Claim 3.4(ii).

So suppose next that $i\in I_{long} \cap [2k+1,4k]$. If $Q_{i,s_{i,t}}$ is an even path for $t=1$ or $t=2$ we take it to be $P_i$.
So suppose that these two paths are odd. We may assume that $q_{i,s_{i,1}}^{5k+5}$ sends an edge to $q_{i,s_{i,2}}^{5k+5}$.
We set $j_{i,1}:=s_{i,1}$ and $j_{i,2}:=s_{i,2}$ and let $P_i$ be as described in Claim~3.4(iii).
If $i\in I_{long} \cap [4k+1,6k]$, we define $P_i$ similarly. This completes the proof of Claim~3.4.

\medskip

We are now ready to prove Claim~3. For each $i\in I_{long}$, let $P_i$ be as given by Claim 3.4.
We will colour all those vertices on the paths $Q_{i,j}$ with
$(i,j)\in I_{long}\times [10^5 k^4 \log(2k)]$ which are uncoloured so far as follows. 

For each $i\in I_{long}\cap [2k]$, we colour all internal vertices of $P_i$ by $\alpha$ if $i\leq k$ and by $\beta$ if $i>k$.  
For each $i\in I_{long} \cap [k]$, we also colour all vertices in $(Q_{i,j}^1\cup Q_{i,j}^9) \setminus (V(P_i)\cup R)$ by $\alpha$ and
all vertices in $(Q_{i,j}^2\cup Q_{i,j}^3\cup Q_{i,j}^7\cup Q_{i,j}^8) \setminus (V(P_i)\cup R)$ by $\beta$ (for all $j\in [10^5 k^4 \log(2k)]$).
Similarly, for each $i\in I_{long} \cap [k+1,2k]$, we colour all vertices in $(Q_{i,j}^1\cup Q_{i,j}^9) \setminus (V(P_i)\cup R)$ by $\beta$
and all vertices in $(Q_{i,j}^2\cup Q_{i,j}^3\cup Q_{i,j}^7\cup Q_{i,j}^8) \setminus (V(P_i)\cup R)$ by $\alpha$.
For each $i\in I_{long}\cap [2k]$, we colour all vertices in $(Q_{i,j}^4\cup Q_{i,j}^6)\setminus (V(P_i)\cup R^{4,6})$ by $\alpha$.

For each $i\in I_{long}\cap [2k+1,6k]$, we colour all internal vertices of $P_i$ in an alternating manner consistent with the endvertices of $P_i$.
(Claim~3.4(iv) ensures that this is possible.) For all $j\in [10^5 k^4 \log(2k)]$ we also colour all vertices in
$Q_{i,j}^0\cup Q_{i,j}^4\cup Q_{i,j}^6 \setminus (V(P_i)\cup R\cup R^{4,6})$ in an alternating manner. (That is, if $b_i=q_{i,j}^0$ is coloured $\alpha$, we colour
$q_{i,j}^s$ by $\alpha$ for all even numbers $s\leq 5k+4$, and colour $q_{i,j}^s$ by $\beta$ for all odd numbers $s\leq 5k+4$.
We colour of each vertex $x$ in $(Q_{i,j}^6\cup \dots \cup Q_{i,j}^9)\setminus (V(P_i)\cup R\cup R^{4,6})$ in a similar way, depending on the colour of $a_i$
and the distance of $x$ to $a_i$ in $Q_{i,j}$.)

Now all uncoloured vertices of $V_{long}$ belong to $Q_{i,j}^5$ for some $i,j$. We let $C^0$ be the union of $V(Q_{i,j}^5)$ over all
$(i,j)\in I_{long}\times [10^5 k^4 \log(2k)]$. We colour all uncoloured vertices in $C^0$ by $\alpha$, and let
$C_4$ denote the set consisting of all the vertices coloured so far.
Note that $|C_4 \setminus C^0| \leq |C_3| + |\overline{V}_{long}| \leq 4\cdot 10^4 k^4\log(2k) +  2\cdot 10^7 k^6\log(2k) \leq 3\cdot 10^7 k^6\log(2k)$. 
Together with Claim~3.1 this implies that parts (i), (ii) and (iii)(b) of Claim~3 hold.

We now show that all the vertices on the paths $Q_{i,j}$ are safe. Together with Claim~3.3(iii) this will imply that all vertices in $C_4$ are safe, i.e.~Claim~3(iii)(a) will
hold. Consider first any vertex $v\in V^0_{long}$. If $v\in R$, then $v$ is safe by Claim 3.3(iii). If $v\notin R$, then by Claim~3.2 $v$ has at least $k$ out-neighbors
and at least $k$ in-neighbours in each of $R_\alpha$ and $R_\beta$, so it has $k$ safe out-neighbours and $k$ safe in-neighbours of each colour.
Thus $v$ is safe by (S3) and (S4). So all the the vertices in $V^0_{long}$ are safe.

Note that if $(i,j)\notin I_{R}$, then $V(Q_{i,j}^1\cup Q_{i,j}^2 \cup Q_{i,j}^3)\setminus \{q_{i,j}^{3k+2}\}$ contains at least $k$ vertices of each colour
and so does $V(Q_{i,j}^7 \cup Q_{i,j}^8 \cup Q_{i,j}^9)\setminus \{q_{i,j}^{|Q_{i,j}|-3k-2}\}$. (Recall that $q_{i,j}^{3k+2}$ is
the final vertex of $Q_{i,j}^3$ and that $q_{i,j}^{|Q_{i,j}|-3k-2}$ is the initial vertex of $Q_{i,j}^7$.)

Now consider a vertex $v\in \overline{V}_{long}\setminus V^0_{long}$, and let $i,j$ be such that $v\in V(Q_{i,j}^4\cup Q_{i,j}^6)$.
If $v\in R^{4,6}$, then $v$ is safe by Claim~3.3(iii). If $v\notin R^{4,6}$, then $(i,j)\notin I_{R}$. But by (Q1) all vertices in
$Q_{i,j}^1\cup Q_{i,j}^2 \cup Q_{i,j}^3$ (apart from possibly the final vertex of $Q_{i,j}^3$) are out-neighbours of $v$,
so $v$ has $k$ safe out-neighbours coloured $\alpha$ and $k$ safe out-neighbours coloured $\beta$.
Similarly, all vertices in $Q_{i,j}^7 \cup Q_{i,j}^8\cup Q_{i,j}^9$ (apart from possibly the initial vertex of $Q_{i,j}^7$) are in-neighbours of $v$,
so $v$ has $k$ safe in-neighbours coloured $\alpha$ and $k$ safe in-neighbours coloured $\beta$. Hence $v$ is safe.

Now consider any vertex $v\in V(Q_{i,j}^5)$. If $(i,j)\notin I_{R}$, a similar argument as above shows that $v$ is safe.
If $(i,j)\in I_{R}$, then by (Q1) all vertices in $Q_{i,j}^4$ (apart from possibly its final vertex) are out-neighbours of $v$, and
all vertices in $Q_{i,j}^6$ (apart from possibly its initial vertex) are in-neighbours of $v$.
Together with Claim~3.3(ii),(iii) this shows that $v$ has $k$ safe out-neighbours and $k$ safe in-neighbours of each colour. So $v$ is safe.
This completes the proof of Claim~3(iii)(a).

To check  Claim~3(iii)(c), note that if a vertex $v$ outside $C_4$ has an out-neighbour in $C^0$, then by (Q3) all vertices in
$Q_{i,j}^1\cup Q_{i,j}^2\cup Q_{i,j}^3\cup Q_{i,j}^4$ (apart from possibly the last two vertices of $Q_{i,j}^4$)
are out-neighbours of $v$. Thus $v$ has at least $k$ out-neighbours of each colour. In a similar way one can use (Q2) to show that
$v$ also has $k$ in-neighbours of each colour. This completes the proof of Claim~3.

\medskip

\noindent {\bf Claim 4:} \textit{We can colour all uncoloured vertices in such a way that every vertex is safe.}

\smallskip
To prove Claim~4, we colour all the vertices outside $C_4$ one by one. We first deal with all vertices in $E_A\setminus C_4$ (see STEP~1),
then we move to the vertices in $E_B\setminus C_4$ (see STEP~2). Finally, we colour all the remaining vertices (see STEP~3).
We let $Z_A:= \emptyset$. While dealing with each vertex in $E_A\setminus C_4$ in turn (i.e.~during STEP~1), we 
will update $Z_A$. At each substep, $Z_A$ will satisfy the following properties:
\begin{itemize}
\item[{\rm (a)}] $Z_A$ consists of coloured vertices and $Z_A\cap (C_4\cup E_A)=\emptyset$,
\item[{\rm (b)}] every coloured vertex lies in $C_4\cup E_A\cup Z_A$,
\item[{\rm (c)}] $|Z_A|\leq 2k|E_A|$.
\end{itemize}

\medskip

\noindent STEP 1. \textit{We can colour all vertices in $E_A\setminus C_4$ as well as some set $Z_A$ of additional vertices in such a way that all the
vertices in $E_A\setminus C_4$ are backwards-safe and alternating-backwards-safe and $Z_A$ satisfies (a)--(c).}

\smallskip  Consider each vertex $v\in E_A\setminus C_4$ in turn. Suppose first that $v$ has $2k$ uncoloured in-neighbours
$v_1,v_2,\dots, v_{2k}$ outside $E_A$. We colour $k$ of them by $\alpha$ and $k$ of them by $\beta$ and
replace $Z_A$ by $Z_A\cup \{v_1,v_2,\dots, v_{2k}\}$. We also colour $v$ with $\alpha$. Note that (S2) implies that $v_1,v_2,\dots, v_{2k}$ are
backwards-safe and alternating-backwards-safe. Together with (S3) and (S4) this shows that $v$ is backwards-safe and alternating-backwards-safe.

So suppose that $v$ has less than $2k$ uncoloured in-neighbours outside $E_A$.
Recall from Claim~3(iii)(b) that at most $3\cdot 10^7k^6\log(2k)$ vertices in $C_4$ lie outside the set~$C^0$.
Together with~(\ref{referee equation E_A}) and (c) this shows that
$$\hat{\delta}^-(T)-|E_A\cup Z_A|  \ge \hat{\delta}^-(T)-3k|E_A| \ge  5\cdot 10^8k^6\log(2k) \ge 2k+|C_4\setminus C^0|.$$
Since all coloured vertices lie in $C_4\cup E_A\cup Z_A$, this implies that $v$ has an in-neighbour in $C^0$. But now Claim~3(iii)(c) implies that $v$ has
$k$ in-neighbours of colour $\alpha$ and $k$ in-neighbours of colour $\beta$ in $C_4$. Since all the vertices in $C_4$ are safe,
this implies that $v$ becomes backwards-safe and alternating-backwards-safe by colouring $v$ with $\alpha$.

\smallskip

Note that we add at most $2k$ vertices to $Z_A$ for each vertex $v\in E_A\setminus C_4$. So at the end of STEP~1,
we will still have that $|Z_A|\leq 2k|E_A|$. Since by (S2) every vertex outside $E_B$ is
forwards-safe and alternating-forwards-safe, after STEP~1, all vertices in $E_A\setminus E_B$ will be safe, while
the vertices in $E_A\cap E_B$ might only be backwards-safe and alternating-backwards-safe.

Let $Z_B:=\emptyset$. While dealing with each vertex in $E_B\setminus C_4$ in turn during STEP~2, we 
will update $Z_B$. At each substep, $Z_B$ will satisfy the following properties (where $Z:=Z_A\cup Z_B$):
\begin{itemize}
\item[{\rm (a$'$)}] $Z_B$ consist of coloured vertices and $Z_B\cap (C_4\cup E\cup Z_A)=\emptyset$,
\item[{\rm (b$'$)}] every coloured vertex lies in $C_4\cup E\cup Z$,
\item[{\rm (c$'$)}] $|Z_B|\leq 2k|E_B|$ and so $|Z|\le 4k|E|$.
\end{itemize}

\medskip

\noindent STEP 2.  \textit{We can colour all uncoloured vertices in $E_B\setminus C_4$ as well as some set $Z_B$ of additional vertices in such a way that all the
vertices in $E_B\setminus C_4$ are safe and $Z_B$ satisfies (a$'$)--(c$'$).}

\smallskip
Consider each vertex $v\in E_B\setminus C_4$ in turn. If $v\notin E_A$, then $v$ is backwards-safe and alternating-backwards-safe by (S2).
If $v\in E_A$, then by STEP~1 $v$ is also backwards-safe and alternating-backwards-safe.

Suppose first that $v$ has $2k$ uncoloured out-neighbours $v_1,v_2,\dots, v_{2k}$ outside $E$. We colour $k$ of them by
$\alpha$ and $k$ of them by $\beta$. We replace $Z_B$ by $Z_B\cup \{v_1,v_2,\dots, v_{2k}\}$. If $v$ is uncoloured, we
colour $v$ with $\alpha$. Then (S2)--(S4) together imply that $v$ becomes safe.

So suppose that $v$ has less than $2k$ uncoloured out-neighbours outside $E$. Note that
$$\hat{\delta}^+(T)-|E\cup Z|  \ge \hat{\delta}^+(T)-5k|E| \ge  5\cdot 10^8k^6\log(2k) \ge 2k+|C_4\setminus C^0|$$
by~(\ref{referee equation E}), (c$'$) and Claim~3(iii)(b). Since all coloured vertices lie in $C_4\cup E\cup Z$,
this implies that $v$ has an out-neighbour in $C^0$. But now Claim~3(iii)(c) implies that $v$ has
$k$ out-neighbours of colour $\alpha$ and $k$ out-neighbours of colour $\beta$ in $C_4$. Since all the vertices in $C_4$ are safe,
this implies that $v$ becomes safe by colouring $v$ with $\alpha$ (in case $v$ is still uncoloured).

\smallskip

Note that we add at most $2k$ vertices to $Z_B$ for each vertex in $E_B\setminus C_4$.
Thus we always have that $|Z_B|\leq 2k|E_B|$ and so $|Z|\leq 4k|E|$.
After STEP~2, all vertices in $C_4\cup E$ are safe.

\medskip

\noindent STEP 3.  \textit{By colouring all the remaining uncoloured vertices with $\alpha$, every vertex becomes safe.}

\smallskip
This follows immediately from~(S2).

\medskip

This completes the proof of Claim~4 and thus of Theorem~\ref{main theorem}. \endproof

\section{Spanning linkedness and non-separating subdivisions} \label{linkage}

The following lemma generalizes a result of Thomassen~\cite{CTt}.
Theorems~\ref{thm:subd} and~\ref{cor:spanningsubd} then both follow easily by an inductive application of Lemma~\ref{path-deletion}.

\begin{lemma} \label{path-deletion}
Let $k,d$ be nonnegative integers.
Let $x,y,z_1,\dots, z_d$ be any distinct vertices in a strongly $(k+d+4)$-connected tournament $T$ and let $P$ be a shortest $xy$-path in $T-\{z_1,\dots,z_d\}$. Then $T-(V(P)\setminus X)$ is strongly $k$-connected for any (possibly empty) subset $X\subseteq\{x,y\}$.
\end{lemma}
\begin{proof}
Write $P:= x_0x_1\ldots x_m$ with $x=x_0$ and $y=x_m$. Note that $P$ must be a backwards-transitive path.
If $P$ has length at most two, the result trivially holds. So suppose that $P$ has length more than two.
Note that in this case it suffices to show that $T-V(P)$ is strongly $k$-connected 
(otherwise we consider $x'\in\{x,x_1\}, y'\in \{y,x_{m-1}\}$ and proceed through the argument with $x',y'$ playing the role of $x,y$).
So suppose $T-V(P)$ is not strongly $k$-connected. 
Then there exist a partition of $V(T) \setminus V(P)$ into nonempty sets $S,S_1,S_2$ such that $|S|\leq k-1$ and no vertex in $S_2$ sends an edge to $S_1$. 
Since $T-(S\cup\{z_1,\dots,z_d\})$ is strongly $5$-connected, there are five paths $P_1,\dots, P_5$ from $S_2$ to $S_1$ which are internally disjoint and do not intersect $S\cup \{z_1,\dots, z_d\}$. 
We may assume that the $P_i$  are backwards-transitive. Moreover, the interior of each $P_i$ is nonempty and is contained in $V(P)$.
Altogether, this means that the intersection of $P_i$ and $T[V(P)]$ is either a segment of $P$ or a path of the form $x_jx_{\ell}$ with $j\geq \ell+2$ or 
of the form $x_jx_{j+1}x_{j-1}$ or $x_{j}x_{j-2}x_{j-1}$. 
We let $p$ be the largest number such that some $P_i$ contains an edge $ux_p$ from $S_2$ to $x_p$ 
and we let $q$ be the smallest number such that some $P_i$ contains an edge $x_qv$ from $x_i$ to $S_1$. 
Note that $p \ge q+4$.
Then the path obtained from $P$ by deleting $x_{q+1}x_{q+2}\dots x_{p-1}$ and adding $x_qvux_p$ is
shorter than $P$, a contradiction.\end{proof}

\removelastskip\penalty55\medskip\noindent{\bf Proof of Theorem~\ref{thm:subd}.}
Write $D=\{w_1,\dots, w_d\}$.
We proceed by induction on $m$. For $m=1$, the assertion holds by Lemma~\ref{path-deletion} applied with $d-2$ playing the role of~$d$.
Suppose that $m \ge 2$ and that the assertion holds for $m-1$. 
Consider any edge $uv \in E(H)$. Without loss of generality, we may assume that $\phi(u)=w_1$ and $\phi(v)=w_2$. 
Then we apply Lemma \ref{path-deletion} (with $d-2$ playing the role of $d$)
to find a $w_1w_2$-path $P$ whose interior does not intersect $D$ and so that $T':=T-V({\rm int}(P))$ is strongly $(k+(m-1)(d+2))$-connected. 
Now by the induction hypothesis, we can find a subdivision $H_*$ of $H\setminus\{uv\}$ in $T'$
which satisfies (i)--(iii) (with $T'$ playing the role of $T$). 
Finally, let $H^* := H_* \cup {\rm int}(P)$.
Then $H^*$ satisfies all requirements.
\endproof

\removelastskip\penalty55\medskip\noindent{\bf Proof of Theorem~\ref{cor:spanningsubd}.}
We proceed by induction on $k$. 
For $k=1$, the assertion was proven by Thomassen~\cite{CTh}. 
Assume that $k \ge 2$ and that the assertion holds for $k-1$.  Let $Z:=\{x_1,\dots, x_{k-1}, y_1,\dots, y_{k-1}\}$ and
let $X:=\{x_k, y_k\} \cap Z$.
We can now apply Lemma \ref{path-deletion} with $d=|Z\setminus X|$ to a find a $x_ky_k$-path $P$ avoiding 
$Z \setminus X$ so that $T[W]$ is strongly $((k-1)^2+3(k-1))$-connected, 
where $W:=V(T) \setminus (V(P)\setminus X)$. 
Now by the induction hypothesis, we can find $P_1, \dots, P_{k-1}$ in $T[W]$ so that $P_i$ is a path from $x_i$ to $y_i$ and
$W = \bigcup_{i=1}^{k-1} V(P_i)$. 
Let $P_k:=P$. Then $P_1,\dots, P_k$ are as desired.
\endproof

\section{Acknowledgements}

We are grateful to J{\o}rgen Bang-Jensen for helpful comments on an earlier version of this paper.

\medskip

{\footnotesize \obeylines \parindent=0pt

Jaehoon Kim, Daniela K\"uhn, Deryk Osthus
School of Mathematics
University of Birmingham
Edgbaston
Birmingham
B15 2TT
UK
}
\begin{flushleft}
{\it{E-mail addresses}:}
{\tt{\{kimJS, d.kuhn, d.osthus\}@bham.ac.uk}}
\end{flushleft}

\end{document}